\documentclass[12pt]{amsart}
\usepackage{amssymb}
\usepackage{amsmath}
\usepackage{longtable}
\newcommand\g{{\mathfrak g}}
\newcommand\h{{\mathfrak h}}

\renewcommand{\b}{\mathfrak{b}}
\newcommand\m{\mathfrak m}
\renewcommand\l{\mathfrak l}
\newcommand\n{\mathfrak n}
\renewcommand\t{\mathfrak t}

\newcommand\z{\mathfrak z}
\newcommand\vf{\mathfrak v}

\newcommand\q{\mathfrak q}

\renewcommand\v{\mathfrak v}
\newcommand\FF{\mathcal F}
\newcommand\GF{\mathcal G}
\newcommand\KF{\mathcal K}

\newcommand\Wh{\operatorname{Wh}}
\newcommand\y{\mathfrak y}

\newcommand\Sk{\mathcal{S}}
\newcommand\Verm{\mathcal{M}}
\newcommand\F{\operatorname{F}}
\newcommand\W{{\bf A}}
\newcommand\K{\mathbb K}
\newcommand\U{\mathcal U}
\newcommand\X{\mathfrak X}

\newcommand\Ann{\operatorname{Ann}}

\newcommand\Mod{\operatorname{Mod}}
\newcommand\Walg{\mathcal W}
\newcommand\Z{\mathbb Z}
\newcommand\A{\mathcal A}

\newcommand\gr{\operatorname{gr}}

\newcommand\I{\mathcal I}
\newcommand\J{\mathcal J}
\renewcommand\sl{\mathfrak{sl}}
\newcommand\Span{\operatorname{Span}}
\newcommand\Hom{\operatorname{Hom}}
\newcommand{\ad}{\mathop{\rm ad}\nolimits}

\newcommand\Centr{\mathcal Z}

\newcommand{\Ocat}{\mathcal{O}}

\newtheorem{Thm}{Theorem}[section]
\newtheorem{Prop}[Thm]{Proposition}
\newtheorem{Cor}[Thm]{Corollary}
\newtheorem{Lem}[Thm]{Lemma}
\theoremstyle{definition}

\newtheorem{Rem}[Thm]{Remark}

\unitlength=1mm   \numberwithin{equation}{section}
\numberwithin{table}{section} \oddsidemargin=0cm
\author{Ivan Losev}
\title{On the structure of the category $\Ocat$ for W-algebras}
\begin{document}
\begin{abstract}
A W-algebra (of finite type) $\Walg$ is a certain associative algebra
associated with a semisimple Lie algebra, say $\g$, and its
nilpotent element, say $e$. The goal of this paper is to study the
category $\Ocat$ for $\Walg$  introduced by Brundan, Goodwin and
Kleshchev. We establish an equivalence of this category with certain
category of $\g$-modules. In the case when $e$ is of principal Levi
type (this is always so when $\g$ is of type A) the category of
$\g$-modules in interest is the category of generalized Whittaker
modules introduced by McDowell, and studied by Milicic-Soergel
and Backelin. 
\end{abstract}
\maketitle
\section{Introduction}
Let $\g$ be a semisimple Lie algebra over an algebraically closed
field $\K$ of characteristic zero. Choose a nilpotent element
$e\in\g$. Associated to the pair $(\g,e)$ is a certain associative
algebra $\Walg$, which is closely related to the universal
enveloping algebra $U(\g)$. It was studied extensively during the
last decade starting from  Premet's paper \cite{Premet1}, see also
\cite{BGK},\cite{BK1},\cite{BK2},\cite{GG},\cite{Ginzburg},\cite{Wquant},\cite{HC},
\cite{Premet2}-\cite{Premet4}. Definitions of
a W-algebra due to Premet, \cite{Premet1}, and the author, \cite{Wquant},
are recalled in Section \ref{SECTION_Walg}.

In the representation theory of $U(\g)$ a crucial role is played by
the Bernstein-Gelfand-Gelfand category $\Ocat$ of $U(\g)$-modules.
In particular, all finite dimensional $U(\g)$-modules and all Verma
modules belong to $\Ocat$. In \cite{BGK} Brundan, Goodwin and
Kleshchev introduced the notion of the category $\Ocat$ for $\Walg$.
This category also contains all finite dimensional $\Walg$-modules
as well as analogs of Verma modules. See Section \ref{SECTION_Ocat}
for definitions.

The BGK category $\Ocat$ is not always very useful. For example, for
a {\it distinguished} nilpotent element $e\in\g$ (i.e., such that
the centralizer $\z_\g(e)$ contains no nonzero semisimple elements) $\Ocat$ consists precisely
of finite dimensional modules. The other extreme
is the case when $e$ is of principal Levi type. This means that
there is a Levi subalgebra $\l\subset\g$ such that $e$ is a
principal nilpotent element in $\l$. Here the BGK category $\Ocat$
looks quite similar to the BGG one.

In \cite{BGK}, Conjecture 5.3, the authors conjectured that for $e$
of principal Levi type there  exists  a category equivalence between
their category $\Ocat$ and a certain category of {\it generalized
Whittaker modules} introduced by McDowell, \cite{McD}, and studied by Milicic and Soergel, \cite{MS}, and
Backelin, \cite{Backelin}. We postpone the description of this
category until Section \ref{SECTION_Milicic_Soergel}. The main
result of this paper, Theorem \ref{Thm_main}, gives the proof of
that conjecture.

Let us describe the content of this paper. In Section
\ref{SECTION_Walg} we recall the definition of W-algebras and
the basic theorem of our paper \cite{Wquant},
the so called decomposition theorem. In Section \ref{SECTION_Ocat} the notion of the
category $\Ocat$ for a W-algebra is recalled. In Section
\ref{SECTION_Milicic_Soergel} we introduce the category of
generalized Whittaker modules. Special cases of this category are, firstly,
Skryabin's category of Whittaker modules (or, more precisely, the full subcategory
there consisting of all finitely generated modules) and, secondly, the categories studied in
\cite{McD},\cite{MS},\cite{Backelin}. Then we state the category equivalence
theorem \ref{Thm_main}  generalizing the Skryabin equivalence theorem
from \cite{Premet1} and proving the conjecture of Brundan, Goodwin
and Kleshchev. The proof of Theorem \ref{Thm_main} is given in
Section \ref{SECTION_proof}. Essentially, we generalize the proof
of the Skryabin equivalence
theorem given in \cite{Wquant}, Subsection 3.3, checking that
certain topological algebras are isomorphic.

Finally, in Section \ref{SECTION_Application} we will describe some applications of
our results.

{\bf Acknowledgements.} I am grateful to Alexander Kleshchev, who
brought this problem to my attention. I also thank Jonathan Brundan for explaining the
application of my results to the classification of representations of Yangians. Finally, I thank the referee
for useful comments on previous versions of this paper that helped to improve the exposition.

\section{W-algebras}\label{SECTION_Walg}
Throughout the paper everything is defined over an algebraically
closed field $\K$ of characteristic 0.

Let $G$ be a  reductive algebraic group, $\g$ its Lie algebra, and
$\U$ the universal enveloping algebra of $\g$. Fix a nonzero
nilpotent element $e\in \g$. Choose an $\sl_2$-triple $(e,h,f)$ in
$\g$ and set $Q:=Z_G(e,h,f)$. Denote by $T$ a maximal torus of $Q$.
Also introduce a grading on $\g$ by eigenvalues of $\ad h$:
$\g:=\bigoplus \g(i), \g(i):=\{\xi\in\g| [h,\xi]=i\xi\}$. Consider
the one-parameter subgroup $\gamma:\K^\times\rightarrow G$ with
$\frac{d}{dt}|_{t=0}\gamma=h$. Choose a $G$-invariant symmetric form $(\cdot,\cdot)$
on $\g$, whose restriction to any algebraic reductive subalgebra is non-degenerate.
This form allows to identify $\g\cong\g^*$. Let $\chi=(e,\cdot)$ be the
element of $\g^*$ corresponding  to $e$.

Equip the space $\g(-1)$ with a symplectic form $\omega_\chi$ as
follows: $\omega_\chi(\xi,\eta)=\langle\chi,[\xi,\eta]\rangle$. Fix
a lagrangian subspace $l\subset \g(-1)$ and define the subalgebra
$\m:=l\oplus \bigoplus_{i\leqslant -2}\g(i)\subset \g$. According to
Premet, \cite{Premet1}, a W-algebra $\Walg$ associated with $e$
is, by definition, $(\U/\U\m_\chi)^{\ad \m}$, where
$\m_\chi:=\{\xi-\langle\chi,\xi\rangle,\xi\in\m\}$. As Gan and
Ginzburg checked in \cite{GG}, $\Walg$ does not depend on the choice
of $l$ up to some natural isomorphism. Thus we can choose  a
$T$-stable lagrangian subspace  $l\subset\g(-1)$ so we get an action of
$T$ on $\Walg$. Note that the image of $\t$ in $\U/\U\m_\chi$
consists of $\ad\m$-invariants, for $\m$ is $\t$-stable and $\chi$ is annihilated by $\t$. So we get an embedding
$\t\hookrightarrow \Walg$.  It is compatible with the action of $T$ in the sense
that the differential of the $T$-action coincides with the adjoint
action of $\t\subset \Walg$. In fact, from the construction in
\cite{GG} it follows that $Q$ acts on $\Walg$ by algebra
automorphisms, see \cite{Premet2}, Subsection 2.2, for details.

One important feature of $\Walg$ is that the category $\Walg\text{-}\Mod$ of (left)
$\Walg$-modules is equivalent to a certain full subcategory in
$\U$-$\Mod$ to be described now. We say that a left $\U$-module $M$ is
a {\it Whittaker} module if $\m_\chi$ acts on $M$ by locally
nilpotent endomorphisms. In this case $M^{\m_\chi}=\{m\in M| \xi
m=\langle\chi,\xi\rangle m, \forall \xi\in\m\}$ is a  $\Walg$-module. As
Skryabin proved in the appendix to \cite{Premet1}, the functor
$M\mapsto M^{\m_\chi}$ is an equivalence between the category of
Whittaker $\U$-modules and $\Walg$-$\Mod$. A quasiinverse equivalence
is given by $N\mapsto \Sk(N):=(\U/\U\m_\chi)\otimes_\Walg N$,
where $\U/\U\m_\chi$ is equipped with a natural structure of a
$\U$-$\Walg$-bimodule.

Note also that the center of $\Walg$ can be identified with the
center $\Centr$ of $\U$, as follows. It is clear that $\Centr\subset
\U^{\ad \m}$ whence we have a homomorphism $\Centr\rightarrow
\Walg$. This homomorphism is injective and its image coincides with
the center of $\Walg$, see \cite{Premet2}, the footnote to the Question 5.1.

An alternative description of $\Walg$ was given in \cite{Wquant}.
Define the Slodowy slice $S:=e+\z_\g(f)$. It will be convenient for
us to consider $S$ as a subvariety in $\g^*$.  Define the {\it
Kazhdan} action of $\K^\times$ on $\g^*$ by
$t.\alpha=t^{-2}\gamma(t)\alpha$. This action preserves $S$ and, moreover,
$\lim_{t\rightarrow \infty}t.s=\chi$ for all $s\in S$. Also note
that $Q$ acts on $S$ in a natural way.

Set $V:=[\g,f]$. Equip $V$ with the symplectic form
$\omega(\xi,\eta)=\langle\chi,[\xi,\eta]\rangle$, the action of
$\K^\times, t.v=\gamma(t)^{-1}v,$ and the natural action of $Q$.

Now let $Y$ be a smooth affine  variety  equipped with
commuting actions of a reductive group $Q$ and of the one-dimensional torus
$\K^\times$. For instance, one can take $Y=\g^*,S,V^*$ equipped with
the natural actions of $Q=Z_G(e,h,f)$ and the Kazhdan actions of
$\K^\times$. Note that the grading on $\K[S]$ induced by the Kazhdan
action is positive.

As follows from the explanation in \cite{HC}, Subsection 2.1, for
$Y=\g^*,V^*,S$ there are certain {\it star-products}
$*:\K[Y]\otimes\K[Y]\rightarrow \K[Y][\hbar], f*g=\sum_{i=0}^\infty
D_i(f,g)\hbar^{2i},$ satisfying the following conditions.

\begin{enumerate}
\item $*$ is associative, that is, a natural extension of $*$ to
$\K[Y][\hbar]$ turns $\K[Y][\hbar]$ into an associative $\K[\hbar]$-algebra, and
$1$ is a unit for this product.
\item $D_0(f,g)=fg$ for all $f,g\in
\K[Y]$.
\item $D_i(\cdot,\cdot)$ is a bidifferential operator of order at
most $i$ in each variable.
\item $*$ is a $Q$-equivariant map $\K[Y]\otimes \K[Y]\rightarrow
\K[Y][\hbar]$.
\item $*$ is homogeneous with respect to $\K^\times$. This, by definition, means
that the degree of $D_i$ is $-2i$ for all $i$.
\item There is a $Q$-equivariant map $\q\rightarrow \K[Y][\hbar],\xi\mapsto \widehat{H}_\xi,$ such that $\hbar^{-2}[\widehat{H}_\xi,
\bullet]$ coincides with the image of $\xi$ under the  differential of the $Q$-action on
$\K[Y][\hbar]$.
\end{enumerate}

This construction allows one to equip $\K[\g^*],\K[V^*],\K[S]$ with
new associative products $*_1$ defined by $f*_1g=\sum_{i=0}^\infty
D_i(f,g)$. The algebras $\K[\g^*],\K[V^*],\K[S]$ with these new
products are $T$ (and, in fact, $Q$)-equivariantly isomorphic to
$\U$, the Weyl algebra $\W_V$ of the vector space $V$, and the W-algebra $\Walg$,
respectively.


We finish this section recalling a decomposition result from
\cite{Wquant},  which plays a crucial role in our
construction.

Recall that if $X$ is an affine algebraic variety and $x$ a point of $X$ we can consider
the completion $\K[X]^\wedge_{x}:=\varprojlim_k \K[X]/\K[X]\m_x^k$, where $\m_x$ denotes the maximal
ideal corresponding to $x$. If $X$ is an affine space, then taking $x$ for the origin
and choosing a basis in $X$, we can identify
$\K[X]^\wedge_x$ with the algebra of formal power series. The algebra $\K[X]^\wedge_x$ is equipped with the
topology of the inverse image. If  $D:\K[X]\otimes \K[X]\rightarrow \K[X]$ is a bidifferential operator,
then it can be uniquely extended to a continuous bidifferential operator $\K[X]^\wedge_x\otimes \K[X]^\wedge_x\rightarrow \K[X]^\wedge_x$.

Since our star-products satisfy (3),  we can extend
them to the completions $\K[\g^*]^\wedge_\chi,
\K[V^*]^\wedge_0,$ $\K[S]^\wedge_\chi$. So we get new algebra structures on
$\K[\g^*]^\wedge_{\chi}[[\hbar]], \K[V^*]^\wedge_0[[\hbar]],\K[S]^\wedge_\chi[[\hbar]]$.
These algebras have unique maximal ideals, for instance, the maximal ideal  $\widetilde{\m}\subset\K[\g^*]^\wedge_\chi[[\hbar]]$
is the inverse image of the maximal ideal in $\K[\g^*]^\wedge_\chi$. The algebra $\K[\g^*]^\wedge_\chi$
is complete in the $\widetilde{\m}$-adic topology. The similar claims hold for the other two algebras.

Consider the algebra $\K[S]^\wedge_\chi[[\hbar]]\otimes_{\K[[\hbar]]}\K[V^*]^\wedge_0[[\hbar]]$
and let $\widetilde{\m}$ denote its maximal ideal corresponding to  the point $(\chi,0)$. Note that
the algebra is not complete in the $\widetilde{\m}$-adic topology. Taking the completion, we get
the {\it completed tensor product}, which we denote by $\K[S]^\wedge_\chi[[\hbar]]\widehat{\otimes}_{\K[[\hbar]]}\K[V^*]^\wedge_0[[\hbar]]$. As a vector space,
the last algebra is just $\K[S\times V^*]^\wedge_{(\chi,0)}[[\hbar]]$.

Finally, note that there is a natural
identification $\varphi: \z_\g(e)\oplus V\rightarrow \g, (\xi,\eta)\mapsto \xi+\eta$.

The first two assertions of the following Proposition follow from \cite{Wquant}, Theorem 3.1.3,
and the third follows from \cite{HC}, Theorem 2.3.1, for semisimple $G$ and from Remark 2.3.2
for a general reductive group $G$.

\begin{Prop}\label{Prop:1}
There is a $Q\times \K^\times$-equivariant isomorphism
$$\Phi_\hbar:\K[\g^*]^\wedge_\chi[[\hbar]]\rightarrow
\K[S]^\wedge_\chi[[\hbar]]\widehat{\otimes}_{\K[[\hbar]]}\K[V^*]^\wedge_0[[\hbar]]$$
of topological $\K[[\hbar]]$-algebras satisfying the following
conditions:
\begin{enumerate}
\item $\Phi_\hbar(\sum_{i=0}^\infty f_i\hbar^{2i})$ contains only
even powers of $\hbar$.
\item The map between cotangent spaces $d_0(\Phi_\hbar)^*:\z_\g(e)\oplus V\rightarrow\g$
induced by $\Phi_\hbar$ coincides with $\varphi$.
\item Let $\iota_1,\iota_2$ denote the embeddings of $\q$ into $\K[\g^*]^\wedge_\chi[[\hbar]],
\K[S]^\wedge_\chi[[\hbar]]\widehat{\otimes}_{\K[[\hbar]]}\K[V^*]^\wedge_0[[\hbar]]$. Then $\Phi\circ\iota_1=\iota_2$.
\end{enumerate}
\end{Prop}


This proposition allows to define a map from the set of two-sided
ideals of $\Walg$ to the analogous set for $\U$. Namely, take a
two-sided ideal $\I\subset \Walg$. As we noted in \cite{Wquant},
Subsection 3.2, there is a unique ideal $\I_\hbar\subset
\K[S][\hbar]$ such that $\I=\I_\hbar/(\hbar-1)$ and $\I_\hbar$ is {\it $\hbar$-saturated}, i.e., $\I_\hbar\cap \hbar
\K[S][\hbar]=\hbar \I_\hbar$. Let $\I^\wedge_\hbar$ denote the closure of
$\I_\hbar$ in $\K[S]^\wedge_\chi[[\hbar]]$. Let $\J_\hbar$ denote
the intersection of
$\Phi_\hbar^{-1}(\I^\wedge_\hbar\widehat{\otimes}_{\K[[\hbar]]}\K[V^*]^\wedge_0[[\hbar]])$
with $\K[\g^*][\hbar]$. Finally, set
$\I^\dagger:=\J_\hbar/(\hbar-1)\subset \U$. By \cite{Wquant},
Proposition 3.4.1 and Theorem 1.2.2(ii),
$\Ann_\U(\Sk(N))=\Ann_\Walg(N)^\dagger$ for any $\Walg$-module
$N$.

\section{Category $\Ocat$ for a W-algebra}\label{SECTION_Ocat}
Recall that we have an embedding $\t\hookrightarrow \Walg$. Also we have a natural embedding
of the cocharacter lattice $\X^*(T):=\Hom(\K^\times,T)$ of $T$ into $\t$. Choose
an  element $\theta\in \X^*(T)\subset \t$. Let $L$
stand for the centralizer of $\theta$ in $G$, this is a Levi
subgroup of $G$. By $\l$ we denote the Lie algebra of $L$, clearly,
$e,h,f\in\l$. Let $T_0$ denote the unit component of $Z(L)\cap T$
and $\t_0$ be the Lie algebra of $T_0$. Note that $\theta\in \t_0$.

The algebra $\Walg$ decomposes into the direct sum of weight spaces
with respect to $\ad\theta$, $\Walg=\sum_{\alpha\in
\Z}\Walg_\alpha$. Set
\begin{equation}\label{eq:2.1}\Walg_{\geqslant
0}:=\bigoplus_{\alpha\geqslant 0}\Walg_\alpha, \Walg_{>
0}:=\bigoplus_{\alpha> 0}\Walg_\alpha, \Walg_{\geqslant
0}^+:=\Walg_{\geqslant 0}\cap \Walg\Walg_{>0}.\end{equation}

It is clear that $\Walg_{\geqslant 0}$ is a subalgebra of $\Walg$,
while $\Walg_{>0},\Walg_{\geqslant 0}^+$ are  two-sided ideals in
$\Walg_{\geqslant 0}$.


Let us make a remark on the choice of generators in the left
$\Walg$-ideal $\Walg\Walg_{>0}$. Now we consider $\Walg$ as
the algebra $\K[S]$ with the modified multiplication. We have an
embedding $\z_\g(e)\hookrightarrow \K[S]$ and $\z_\g(e)$ generates
$\K[S]$ (for both multiplications). It is easy to see that
$\z_\g(e)_{>0}$ generates $\Walg\Walg_{>0}$. Let $f_1,\ldots,f_n$ denote
a homogeneous (w.r.t. $\ad\theta$) basis in $\z_\g(e)_{>0}$.

Set $\Walg^0:=\Walg_{\geqslant 0}/\Walg_{\geqslant 0}^+$.

\begin{Rem}\label{Rem:2.1}
In \cite{BGK} Brundan, Kleshchev and Goodwin constructed an isomorphism
between $\Walg^0$ and the W-algebra $\underline{\Walg}$ constructed for the
pair $(\l,e)$. We will also show that there is an isomorphism between the two algebras
in the course of the proof of Theorem \ref{Thm_main}. See also Remark \ref{Rem_two_isomorphisms}.
\end{Rem}

%

Proceed to the definition of  full subcategories
$\widetilde{\Ocat}(\theta), \widetilde{\Ocat}^{\t_0}(\theta),
\Ocat(\theta),\Ocat^{\t_0}(\theta)$ in the category $\Walg$-$\Mod$ of left
$\Walg$-modules.

First, we say that a $\Walg$-module $M$ belongs to
$\widetilde{\Ocat}(\theta)$ if $M$ is finitely generated and the following
condition holds:

\begin{itemize}
\item[(*)] for any $m\in M$ there exists $\alpha\in \Z$ such that $\Walg_\beta
m=0$ for any $\beta\geqslant \alpha$.
\end{itemize}

Clearly, $\widetilde{\Ocat}(\theta)$ is an abelian subcategory in the category
$\Walg$-$\Mod$. By definition, $\widetilde{\Ocat}^{\t_0}(\theta)$
consists of all modules  in $\widetilde{\Ocat}(\theta)$, where the
action of $\t_0$ is diagonalizable.

For  example,  take a finitely generated $\Walg^0$-module
$V$. Set $\Verm^\theta(V):=\Walg\otimes_{\Walg_{\geqslant 0}}V$, where
$\Walg_{\geqslant 0}$ acts on $V$ via an epimorphism
$\Walg_{\geqslant 0}\twoheadrightarrow \Walg^0$. The
module $\Verm^\theta(V)$ lies in $\widetilde{\Ocat}(\theta)$. It belongs
to $\widetilde{\Ocat}^{\t_0}(\theta)$ iff the action of $\t_0$ on
$V$ is diagonalizable.

\begin{Lem}\label{Lem:2.2}
For a finitely generated $\Walg$-module $M$ the condition $(*)$ is
equivalent to either of the following two  conditions:
\begin{itemize}
\item[(*$'$)]
$\Walg_{>0}$ acts on $M$ by locally nilpotent endomorphisms.
\item[(*$''$)] Elements $f_i\in \Walg_{>0}$  act on $M$ by locally nilpotent endomorphisms.
\end{itemize}
\end{Lem}
\begin{proof}
Clearly, $(*)\Rightarrow (*')\Rightarrow (*'')$. Let us prove the
implication $(*'')\Rightarrow (*)$. Let $m_1,\ldots,m_r$ generate
$M$. Let $N$ be such that $f_i^N m_j=0$ for all $i,j$. Let $\I$ denote the intersection of
$\Walg_{\geqslant 0}$ with the left
ideal in $\Walg$ generated by $f_i^N$. It is easy to
see that $\gr(\I)$ contains $\K[S]_{>\beta}$ for sufficiently large
$\beta$. Since both $\I$  and the filtration are
$\ad\theta$-stable, we see that $\Walg_{>\beta}\subset \I$. So
$\Walg_{>\beta}$ annihilates $m_1,\ldots,m_r$. This condition
implies (*).
\end{proof}

To define the two remaining subcategories we need a certain functor
$\widetilde{\Ocat}(\theta)\rightarrow \Walg^0$-$\Mod$. This
is the functor of taking $\Walg_{>0}$-invariants. More precisely,
set
$$\FF(M)=M^{\Walg_{>0}}:=\{m\in M| w.m=0, \forall w\in
\Walg_{>0}\}.$$

Note that the functor $\Verm^{\theta}:\Walg^0$-$\Mod\rightarrow \widetilde{\Ocat}(\theta)$
is left adjoint to $\FF:\widetilde{\Ocat}^\theta\rightarrow \Walg^0\text{-}\Mod$. Indeed,
$$\Hom_{\Walg}(\Walg\otimes_ {\Walg_{\geqslant 0}}N,M)=\Hom_{\Walg_{\geqslant 0}}(N,M)=\Hom_{\Walg^0}(N,\FF(M)), N\in \Walg^0\text{-}\Mod, M\in \widetilde{\Ocat}(\theta).$$

By definition, $\Ocat(\theta)$ (resp., $\Ocat^{\t_0}(\theta)$)
consists of all modules $M$ in $\widetilde{\Ocat}(\theta)$ (resp.,
$\widetilde{\Ocat}^{\t_0}(\theta)$) with $\dim\FF(M)<\infty$. Note
that all finite dimensional $\Walg$-modules lie in $\Ocat(\theta)$.

Let us state some results describing the properties of our four
categories.

\begin{Prop}\label{Prop:2.2}
\begin{enumerate}
\item The action of $\t_0$ on any module from $\Ocat(\theta)$
is locally finite.
\item Any  module in $\Ocat(\theta)$ contains a submodule from $\Ocat^{\t_0}(\theta)$.
\end{enumerate}
\end{Prop}
\begin{proof}
The subspace $\FF(M)\subset M$ is finite dimensional.
 Let $\FF(M)_{diag}$ be the sum of all
weight subspaces for $\t_0$ in $\FF(M)$. Then $\FF(M)_{diag}$ is a $\Walg^0$-submodule in
$\FF(M)$. Let $M_0$ be the image of $\Verm^\theta(\FF(M)_{diag})$ in $M$ under
the natural homomorphism. Then the action of $\t_0$ on $M_0$ is diagonalizable.
Hence the second assertion.

Since $M$ is a Noetherian $\Walg$-module, we see that there is a filtration $M_0\subset M_1\subset \ldots\subset
M_k=M$ such that the action of $\t_0$ on every quotient $M_i/M_{i-1}$  is locally finite.
Assertion 1 follows.
\end{proof}

\begin{Prop}\label{Prop:2.3}
Let $M\in \widetilde{\Ocat}(\theta)$. Then the following conditions
are equivalent:
\begin{enumerate}
\item $M\in \Ocat(\theta)$.
\item For any $\alpha\in \K$ the root subspace $M^\alpha:=\bigcup_{i}\ker(\theta-\alpha)^i$
is finite dimensional.
\item For any $\widetilde{\alpha}\in \t_0^*$ the root subspace $M^{(\widetilde{\alpha})}:=\bigcup_{i}\left(\bigcap_{\xi\in\t_0}
\ker(\xi-\langle\widetilde{\alpha},\xi\rangle)^i\right)$ is finite dimensional.
\end{enumerate}
\end{Prop}

In
the proof of Proposition \ref{Prop:2.3} we will need the following simple lemma.

\begin{Lem}\label{Lem:2.4}
Define the partial order on $\K$ as follows: $x\preceq y$ if $y-x$
is a nonnegative integer. We write $x\prec y$ if
$x\preceq y, x\neq y$. For any $M\in \widetilde{\Ocat}(\theta)$
there exist $\alpha_1,\ldots,\alpha_r\in \K$ such that for any
eigenvalue $\beta$ of $\theta$ on $M$ there exists $i$ with
$\beta\preceq \alpha_i$.
\end{Lem}
\begin{proof}
Take finite number of homogeneous generators of $M$ and use the
condition (*).
\end{proof}

Proposition \ref{Prop:2.3} and Lemma \ref{Lem:2.4} show that for $\t=\t_0$ the category
$\Ocat^{\t}(\theta)$ consists of the same modules as the category $\Ocat(e)$ studied in \cite{BGK}: for example, the implication (1)$\Rightarrow$(3) shows that any module from $\Ocat^{\t}(\theta)$ lies in $\Ocat(e)$.
Recall that in \cite{BGK} the category $\Ocat(e)$ was defined as the full subcategory in
$\Walg$-$\Mod$ consisting of all modules, where $\t$ acts diagonalizably, all weight subspaces are finite dimensional,
and the set of weights is bounded from above. Their notion of being "bounded from above" is equivalent to
that mentioned in Lemma \ref{Lem:2.4} although is stated in a different way.

\begin{proof}[Proof of Proposition \ref{Prop:2.3}]
Let us check $(1)\Rightarrow (2)$. Assume the converse. Choose
maximal (w.r.t $\preceq$) $\alpha\in \K$ such that $\dim
M^\alpha=\infty$. Let $f_1,\ldots,f_n$ be such as above. Then, by the choice of $\alpha$, we see that $\dim
f_i M^\alpha<\infty$ for any $i$.  We see that $f_i
M^{\alpha}\subset \bigoplus_{\alpha\prec \beta} M^{\beta}$ for any
$i$. It follows that the intersection of the kernels of $f_i$ in
$M^\alpha$ is infinite dimensional. But this intersection coincides
with $M^\alpha\cap\FF(M)$. Contradiction.

The implication $(2)\Rightarrow (3)$ follows from $M^{(\widetilde{\alpha})}\subset M^{\langle\widetilde{\alpha},\theta\rangle}$. So it remains to check that
$(3)\Rightarrow (1)$.

For any $\widetilde{\alpha}\in \t_0^*$ the weight subspace  $$\FF(M)_{\widetilde{\alpha}}:=\{v\in \FF(M)|\xi.v=\langle\widetilde{\alpha},\xi\rangle v, \forall \xi\in\t_0 \}$$
is a finite dimensional $\Walg^0$-submodule. Let $M_0$ denote the submodule
in $M$ generated by $\FF(M)_{\widetilde{\alpha}}$. Clearly, $M_0$ is isomorphic
to the quotient of $\Verm^\theta(\FF(M)_{\widetilde{\alpha}})$.  If $M$ is irreducible,
then $M=M_0$. In particular, $\langle\widetilde{\beta},\theta\rangle\prec\langle\widetilde{\alpha},\theta\rangle$
for any $\t_0$-weight $\widetilde{\beta}$ of $\FF(M)$. It follows that $M\in \Ocat(\theta)$.
In general, since $\FF$ is a left exact functor, it remains to check that $M$ has finite length.

 As Brundan, Goodwin
and Kleshchev proved in \cite{BGK}, Corollary 4.11, the module $\Verm^\theta(V)$ has finite length
provided $V$ is finite dimensional and irreducible. Actually, they considered the case when $\t=\t_0$ but their proof
extends to the general case directly. Since $\Verm^\theta$ is a right exact functor, we see that $\Verm^\theta(\FF(M)_{\widetilde{\alpha}})$ has finite length. Thus $M_0$ has finite length.
Finally, since $M$ is Noetherian, we see that $M$ has finite length.
\end{proof}

From this proposition and its proof we deduce the following

\begin{Cor}\label{Cor:2.6}
The subcategory $\Ocat(\theta)$ is a Serre subcateogry in
$\widetilde{\Ocat}(\theta)$ (i.e., it is closed w.r.t. taking subquotients and
extensions) and any module in $\Ocat(\theta)$ has
finite length. Furthermore, $\Ocat(\theta)$ contains all Verma modules $\Verm^\theta(V)$
for finite dimensional $V$.
\end{Cor}

As Brundan, Goodwin and Kleshchev noticed in \cite{BGK} (in the case
$\t=\t_0$, the general case is completely analogous), the following
statement holds.

\begin{Prop}\label{Prop:2.8}
Let $V$ be an irreducible finite dimensional
$\Walg^0$-module.  There is a unique simple quotient
$L^\theta(V)$ of $\Verm^\theta(V)$ and any simple module in
$\Ocat(\theta)$ is isomorphic to some $L^\theta(V)$.
\end{Prop}


\begin{Rem}
Note that $\widetilde{\Ocat}(\theta)$ depends on the choice of $\theta$.
However, it follows from Lemma \ref{Lem:2.2} that two categories
$\widetilde{\Ocat}(\theta_1),\widetilde{\Ocat}(\theta_2)$ consist of the same modules
provided the spaces $\z_\g(e)_{>0}$ constructed for
$\theta_1,\theta_2$ coincide. So we get only finitely many different
categories $\Ocat(\theta)$.
\end{Rem}

\section{Generalized Whittaker
modules}\label{SECTION_Milicic_Soergel} In this section we introduce
a certain category of $\U$-modules generalizing Whittaker modules
mentioned in Section \ref{SECTION_Walg}. Also we state our main
result here.

Recall that we have fixed  $\theta\in \X^*(T)\cap\t_0$. Let $
\g_{>0}$ denote the sum of all eigenspaces  for $\ad\theta$ with
positive eigenvalues. Then $\g_{\geqslant 0}:=\l\oplus \g_{>0}$ is a
parabolic subalgebra of $\g$ and $\g_{>0}$ is its nilpotent radical.

Recall the grading $\g=\bigoplus_{i}\g(i)$ introduced in the
beginning of Section \ref{SECTION_Walg}. For $\l(i):=\l\cap\g(i)$ we
have $\l=\bigoplus_{i}\l(i)$.  By analogy with the subalgebra
$\m\subset \g$ define a subalgebra $\underline{\m}\subset\l$ and its
shift $\underline{\m}_\chi$ so that we can define the W-algebra
$\underline{\Walg}=(\underline{\U}/(\underline{\U}\underline{\m}_\chi))^{\ad\underline{\m}}$, where $\underline{\U}:=U(\l)$.
Finally, set $\widetilde{\m}=\underline{\m}\oplus \g_{>0},
\widetilde{\m}_\chi:=\{\xi-\langle\xi,\chi\rangle, \xi\in
\widetilde{\m}\}$.

Let $M$ be a finitely generated left $\U$-module. We say that $M$ is
a {\it generalized  Whittaker module} for $(e,\theta)$ if  $\widetilde{\m}_\chi$ acts
on $M$ by locally nilpotent endomorphisms. The full subcategory of
$\U$-$\Mod$ consisting of all generalized Whittaker modules for
$(e,\theta)$ will be denoted by $\widetilde{\Wh}(e,\theta)$. By
$\widetilde{\Wh}^{\t_0}(e,\theta)$ we denote the  full subcategory
in $\widetilde{\Wh}(e,\theta)$ consisting of all modules with
diagonalizable action of $\t_0$.

For example, let $V$ be a $\underline{\Walg}$-module. Set
$$\Verm^{e,\theta}(V)=\U\otimes_{U(\g_{\geqslant 0})}\Sk_\l(V), $$
where $\Sk_\l(V)=(\underline{\U}/\underline{\U}\underline{\m}_\chi)\otimes_{\underline{\Walg}}V$.
 This module always lies in
$\widetilde{\Wh}(e,\theta)$. It lies in
$\widetilde{\Wh}^{\t_0}(e,\theta)$ iff the action of $\t_0$ on $V$
is diagonalizable.

Let us construct a functor from $\widetilde{\Wh}(e,\theta)$ to the
category $\underline{\Walg}$-$\Mod$. Set
$\GF(M):=M^{\widetilde{\m}_\chi}$. The algebra $\underline{\U}$
acts naturally on $M^{\g_{>0}}$. Since
$\GF(M)=(M^{\g_{>0}})^{\underline{\m}_\chi}$, there is a natural action of
$\underline{\Walg}$  on $\GF(M)$.  It is clear that $\GF(M)\neq
\{0\}$ provided $M\neq \{0\}$. We say that $M\in
\widetilde{\Wh}(e,\theta)$ is of {\it finite type} if $\dim
\GF(M)<\infty$. The category of all finite type modules is denoted
by $\Wh(e,\theta)$. Finally, set
$\Wh^{\t_0}(e,\theta)=\Wh(e,\theta)\cap
\widetilde{\Wh}^{\t_0}(e,\theta)$.

Note also that, analogously to the previous section, the functor $\Verm^{e,\theta}:\underline{\Walg}\text{-}\Mod
\rightarrow \widetilde{\Wh}(e,\theta)$ is left adjoint to $\GF$.

The following theorem is the main result of the paper.

\begin{Thm}\label{Thm_main}
There is an equivalence $\KF:\widetilde{\Wh}(e,\theta)\rightarrow
\widetilde{\Ocat}(\theta)$ of abelian categories and an isomorphism $\Psi:\underline{\Walg}\rightarrow
\Walg^0$ satisfying the
following conditions:
\begin{enumerate}
\item $\Ann_{\Walg}(\KF(M))^{\dagger}=\Ann_{\U}(M)$ for any $M\in \Wh(e,\theta)$.
\item $\KF$ maps $\widetilde{\Wh}^{\t_0}(e,\theta)$ to
$\widetilde{\Ocat}^{\t_0}(\theta)$, and $\Wh(e,\theta)$ to $\Ocat(\theta)$.
\item The functors $\Psi^*\circ\FF\circ \KF$ and $\GF$ from $\widetilde{\Wh}(e,\theta)$
to $\underline{\Walg}$-$\Mod$ are isomorphic. Here $\Psi^*$ denotes the pull-back functor
between the categories of modules induced by $\Psi$.
\item The functors $\KF\circ\Verm^{e,\theta},\Verm^\theta\circ \Psi^{-1*}$ from $\underline{\Walg}\text{-}\Mod$
to $\widetilde{\Ocat}(\theta)$ are isomorphic.
\end{enumerate}
\end{Thm}

In particular, $\KF$ induces an equivalence of abelian categories
$\Wh(e,\theta)\rightarrow \Ocat(\theta)$. Moreover,  for any
irreducible finite dimensional $\underline{\Walg}$-module $V$ the
$\U$-module $\Verm^{e,\theta}(V)$ has a unique simple quotient
$L^{e,\theta}(V)$. The last claim follows either from the
theorem above or can be proved in the same way as an analogous
statement in \cite{MS}, Proposition 2.1.

Now we consider an important special case. Till the end of the
section we assume that $e$ is regular in $\l$. The results below in this section will not be used in the proof
of Theorem \ref{Thm_main}.

The following (quite standard) proposition shows that the category
$\Wh(e,\theta)$ coincides with the category considered in
\cite{MS},\cite{Backelin}.

\begin{Prop}
Let $M\in \widetilde{\Wh}(e,\theta)$, where $e$ is regular in $\l$.
Then the following two conditions are equivalent:
\begin{enumerate}
\item $\dim \GF(M)<\infty$.
\item The action of the center $\Centr$ of $\U$ on $M$ is locally finite.
\end{enumerate}
\end{Prop}
\begin{proof}
$(1)\Rightarrow (2)$. Thanks to Theorem \ref{Thm_main} and Corollary \ref{Cor:2.6}, we may assume that
$M$ is irreducible. Thence the natural homomorphism $\Verm^{e,\theta}(\GF(M))\rightarrow M$
is surjective. Note that $\GF(M)$ is $\Centr$-stable. Since $\dim \GF(M)<\infty$, the $\Centr$-action
 on $\GF(M)$ is locally finite.  But $\GF(M)$ generates $M$ hence the $\Centr$-action on $M$
  is locally finite.

$(2)\Rightarrow (1)$. According to \cite{MS}, Theorem 2.6, $M$ has finite length. So we may assume that
$M$ is irreducible. Let $u,v\in \GF(M)$ be eigenvectors for $\theta$
with eigenvalues $\alpha,\beta$. We may assume that $\alpha\not\preceq\beta$. However, all
eigenvalues in $\U v$ are $\preceq \beta$, for $v$ is annihilated by $\g_{>0}$. So $u\not \in\U v$, contradiction.
\end{proof}

Clearly, $\widetilde{\m}$ is a
maximal subalgebra of $\g$ consisting of nilpotent elements. Let
$\b:=\n_\g(\widetilde{\m})$ be the corresponding Borel subalgebra,
$\h\subset\b$ a Cartan subalgebra. Let $\Delta,\Delta_+,\Pi$ be,
respectively, the root system and the sets of positive roots and of
simple roots corresponding to $(\b,\h)$. Let $\Pi'\subset \Pi$
consist of all simple roots $\alpha$ such that
$\chi|_{\g_\alpha}\neq 0$. Then $\Pi'$ is the system of simple roots
in $\l$. Conjugating $\chi$ by an element of $B\cap L$, where $B$ is the Borel
subgroup of $G$
corresponding to $\b$, if necessary, we may assume that $\chi|_{\g_\alpha}=0$
for $\alpha\in \Delta_+\setminus\Pi'$.  Finally, let $W'\subset W$ be the Weyl group of $\l$.
As Kostant proved in \cite{Kostant},
$\underline{\Walg}$ is identified with the center of  $\underline{\U}$. So any irreducible
$\underline{\Walg}$-module is one-dimensional. These irreducible modules are parametrized
by $W'$-orbits in $\h^*$ for the action given by $w\cdot
\lambda=w(\lambda+\rho)-\rho$, where, as usual,
$\rho=\frac{1}{2}\sum_{\alpha\in \Delta_+}\alpha$. So we get "Verma
modules" $\Verm^{e,\theta}(\lambda)$ with
$\Verm^{e,\theta}(\lambda_1)=\Verm^{e,\theta}(\lambda_2)$ iff
$\lambda_1,\lambda_2$ are $W'$-conjugate, and their simple quotients
$L^{e,\theta}(\lambda)$. As Milicic and Soergel checked in
\cite{MS}, Proposition 2.1, Theorem 2.6, any simple module in
$\widetilde{\Wh}(e,\theta)$ is isomorphic to
$L^{e,\theta}(\lambda)$, and $L^{e,\theta}(\lambda_1)\cong
L^{e,\theta}(\lambda_2)$ iff $\lambda_1,\lambda_2$ are
$W'$-conjugate.

Since any module in $\widetilde{\Wh}(e,\theta)$ (in particular,
$\Verm^{e,\theta}(\lambda)$) has finite length, the multiplicity
$[\Verm^{e,\theta}(\lambda):L^{e,\theta}(\mu)]$ is defined. Theorem 6.2
in \cite{Backelin} reduces the computation of this multiplicity to a
similar problem in the usual BGG category $\Ocat$. Let
$M(\lambda),L(\lambda)$ denote the Verma module and the irreducible
module with highest weight $\lambda\in \h^*$.

\begin{Thm}[Backelin]\label{Thm:4.2}
Let $\lambda,\mu\in \h^*$. If
\begin{enumerate}\item $\mu$ and $\lambda$ are $W$-conjugate
(w.r.t. the $\cdot$-action),
\item  and there is $w\in W'$ such that
 $w\cdot \mu$ is antidominant
for $\l$ (i.e., $\langle w\cdot\mu+\rho,\alpha^\vee\rangle\not\in \Z_{>
0}$ for any $\alpha\in \Delta_+$ with
$\langle\alpha,\theta\rangle=0$), and $\lambda-w\cdot \mu\in \Span_{\Z_{\geqslant 0}}(\Delta^+)$,
\end{enumerate}
then
 $[\Verm^{e,\theta}(\lambda):
L^{e,\theta}(\mu)]=[M(\lambda):L(w\cdot\mu)]$. Otherwise,
$[\Verm^{e,\theta}(\lambda): L^{e,\theta}(\mu)]=0$.
\end{Thm}

Thanks to Theorem \ref{Thm_main}, Theorem \ref{Thm:4.2} allows to compute the decomposition numbers for the
category $\Ocat(\theta)$.

\section{Proof of the main theorem}\label{SECTION_proof}
The proof of Theorem \ref{Thm_main} is based on a construction of completions
from \cite{Wquant}, Subsection 3.2. Let us recall this construction here.

Let $\vf$ be a finite dimensional graded vector space,
$\vf=\bigoplus_{i\in \Z} \vf(i)$, $\vf\neq\vf(0)$, and $A:=S(\vf)$.
Suppose also that a torus $T_0$ acts on $\vf$ preserving the grading.
The grading on $\vf$ gives rise to the grading $A=\bigoplus_{i\in
\Z}A(i)$ whence to the action $\K^\times:A$. Let $*:A\otimes A\rightarrow A[\hbar^2]$ be a
$T_0$-invariant homogeneous star-product, $f*g=\sum_{i=0}^\infty
D_i(f,g)\hbar^{2i}$. Suppose $D_i:A\otimes A\rightarrow A$ is a
bidifferential operator  of order at most $i$ at each variable. Then
we can form the associative product $\circ:A\times A\rightarrow A,
f\circ g=\sum_{i=0}^\infty D_i(f,g)$. We denote $A$ equipped with
the corresponding algebra structure by $\A$.

We have an action of $T_0$ on $\A$ by algebra automorphisms. Suppose that we have an embedding
$\t_0\hookrightarrow \A$ such that the differential of the $T_0$-action on $\A$
coincides with the adjoint action of $\t_0$.

For $u,v\in \vf(1)$ denote by $\omega_1(u,v)$ the constant term of
$u\circ v-v\circ u$. Choose a $T_0$-invariant maximal isotropic (w.r.t. $\omega_1$)
subspace $\y\subset \vf(1)$ and set $\m:=\y\oplus
\bigoplus_{i\leqslant 0}\vf(i)$. Further, choose a homogeneous  basis
$v_1,\ldots,v_n$ of $\vf$  such that $v_1,\ldots,v_m$ form a basis
in $\m$. Let $d_i$ denote the degree of $v_i$. We may assume that the sequence $d_1,\ldots,d_m$
is increasing and all vectors $v_i$ are $T_0$-semiinvariant.

By $A^\heartsuit$ we denote the subalgebra of the formal power
series algebra $\K[[\vf^*]]$ consisting of all formal power series
of the form $\sum_{i<n}f_i$ for some $n$, where $f_i$ is a
homogeneous power series of degree $i$. For any $f,g\in
A^\heartsuit$ we have the well-defined element $f\circ
g:=\sum_{i=0}^\infty D_i(f,g)\in A^\heartsuit$. The space
$A^\heartsuit$ considered as an algebra w.r.t. $\circ$ is denoted by $\A^\heartsuit$.
Any element   $a\in\A^\heartsuit$
can be written in a unique way as an infinite linear combination  $\widetilde{a}$
of monomials
\begin{equation}\label{eq:7.1} v_{i_1}\circ\ldots \circ v_{i_l} \text{ with } i_1\geqslant
i_2\geqslant \ldots\geqslant i_l\end{equation} such that
$\sum_{j=1}^ld_{i_j}\leqslant c$, where $c$ depends on $a$. Let $\F_c\A$ denote
the subspace consisting of all elements, where degrees of monomials are bounded by $c$.
Then the subspaces $\F_c\A^\heartsuit$ form a filtration of $\A^\heartsuit$.

 Pick $\theta\in \X^*(T_0)\subset\t_0$. Let $\vf_{\geqslant 0},\vf_{>0}$ denote the sums of $\ad\theta$-eigenspaces corresponding, respectively, to nonnegative and positive eigenvalues. Let $\A_{\geqslant 0},\A_{>0},\A^\heartsuit_{\geqslant 0},
\A^\heartsuit_{>0}$ be defined analogously (although the action of $\ad\theta$ on $\A^\heartsuit$ is not diagonalizable, the last definition makes sense).
Suppose that the eigenvalues of $v_1,\ldots,v_n$ are decreasing, and $\vf_{>0}\subset \m\subset \vf_{\geqslant 0}$.
Then $\A_{\geqslant 0}^+:=\A_{\geqslant 0}\cap
\A\A_{>0}, \A_{\geqslant 0}^{\heartsuit+}:=\A^\heartsuit_{\geqslant 0}\cap
\A^\heartsuit\A^\heartsuit_{>0}$ are two-sided ideals in $\A_{\geqslant 0},\A^\heartsuit_{\geqslant 0}$.
Set $\A^0:=\A_{\geqslant 0}/\A_{\geqslant 0}^+, \A^{\heartsuit 0}:= \A_{\geqslant 0}^\heartsuit/\A_{\geqslant 0}^{\heartsuit +}$. Note that there is a natural inclusion $\A^0\hookrightarrow \A^{\heartsuit 0}$.

Clearly, an element of $\A^0$ (resp., of $\A^{\heartsuit 0}$) may be thought as a finite (resp., infinite
with finiteness condition stated after (\ref{eq:7.1})) sum of monomials  in $v_i\in \v_0$. Also let us note that $\A^0$ is obtained
from $S(\v_0)$ in the same way as $\A$ is obtained from $S(\v)$ (i.e., using a star-product with properties
listed in the beginning of the section). So we can construct the algebra $\A^{0\heartsuit}$. However, it is clear
that $\A^{0\heartsuit}=\A^{\heartsuit 0}$.

Consider the space $\A^\wedge:=\varprojlim \A/\A\m^k$. It follows from \cite{Wquant} (Lemma 3.2.8
and the discussion before it) that $\A^\wedge$ has a natural structure of a topological algebra
such that the natural map $\A\rightarrow \A^\wedge$ is an algebra homomorphism. Moreover,
this map is injective and extends to an injective homomorphism of algebras
$\A^\heartsuit\rightarrow \A^\wedge$. The algebra $\A^\wedge$ consists of all infinite sums of
monomials (\ref{eq:7.1}) satisfying the following condition:

 for any given $j\geqslant 0$ there are only finitely many monomials
with nonzero coefficients and $v_{i_{l-j}}\not\in\m$.

Furthermore, we can compare algebras of the form $\A^\wedge$ for two different star-products.
The following result follows from \cite{Wquant}, Lemmas 3.2.8,3.2.9.

\begin{Prop}\label{Prop:7.2}
Let $\vf,A$ be as in the beginning of this section and $*,*'$ be two $*$-products on $A$
satisfying the above conditions. So we get new products $\circ,\circ'$ on $A$, the corresponding algebras
will be denoted by $\A,\A'$. Suppose there is a $T_0$-stable subspace $\y\subset\vf(1)$  that
is maximal isotropic
for both skew-symmetric forms.  Finally, suppose that any element in $A$ can be represented as a
finite sum of monomials (\ref{eq:7.1}) and also as a finite sum of analogous monomials for $\circ'$. Suppose there
is a homogeneous $T_0$-equivariant  isomorphism $\Phi:\A^\heartsuit\rightarrow \A'^\heartsuit$ such that $\Phi(v_i)-v_i\in
\F_{d_i-2}\A+(\F_{d_i}\A\cap \vf^2 A)$. Then $\Phi$ extends uniquely
to a topological algebra isomorphism $\Phi:\A^\wedge\rightarrow \A'^\wedge$ with $\Phi(\A^\wedge\m)=\A'^\wedge\m$.
\end{Prop}

Clearly, $\Phi$ induces an isomorphism $\A^{0\heartsuit}=\A^{\heartsuit 0}\rightarrow \A'^{\heartsuit 0}=\A'^{0\heartsuit}$, which is  denoted by $\Phi^0$. This isomorphism is extended
to an isomorphism $\Phi^0:\A^{0\wedge}\rightarrow \A'^{0\wedge}$, where $\A^{0\wedge}:=\varprojlim \A^0/\A^0\m_0^k$,
$\m_0:=\m\cap\v_0$ and $\A'^{0\wedge}$ is defined analogously. Again, we have  $\Phi^0(\A^{0\wedge}\m_0)=\A'^{0\wedge}\m_0$.

Now we will discuss a certain category of $\A^\wedge$-modules. Namely, we consider
topological $\A^\wedge$-modules $M$ equipped with discrete topology. This means that any
vector in $M$ is annihilated by some neighborhood of zero, i.e., by some  $\A^\wedge\m^k$.
So this category is the same as the category of $\A$-modules, where $\m$ acts by locally nilpotent
endomorphisms. This category is denoted by $\widetilde{\Wh}(A,\m)$. Also we need its subcategory
$\widetilde{\Wh}^{\t_0}(A,\m)$ consisting of all modules with diagonalizable action of $\t_0$.

In particular, we have the following straightforward corollary of Proposition \ref{Prop:7.2}.

\begin{Cor}\label{Cor:7.3}
Preserve the assumptions of Proposition \ref{Prop:7.2}. Then there are equivalences $\Phi_*:\widetilde{\Wh}(\A,\m)\rightarrow\widetilde{\Wh}(\A',\m), \widetilde{\Wh}^{\t_0}(\A,\m)\rightarrow\widetilde{\Wh}^{\t_0}(\A',\m),
\Phi^0_*:\widetilde{\Wh}(\A^0, \m_0)\rightarrow \widetilde{\Wh}(\A'^0,\m_0)$ induced by $\Phi$ and $\Phi^0$. Moreover,
\begin{equation}\label{eq:7.6}\Phi_*(M^{\m})=\Phi_*(M)^{\m}\end{equation} for any $\A$-module $M$ (with locally nilpotent
action of $\m$).
\end{Cor}

Let us specify now $\A,\A',\m$.

The torus $T_0$ we are going to consider is the same as in Section \ref{SECTION_Ocat}.
Fix  $m\in \mathbb{N}, m>2+2d$, where $d$ denotes the maximal eigenvalue of $\ad h$ in $\g$.
Consider the diagonal embedding $\K^\times \hookrightarrow\K^\times\times
T_0$, whose differential is given by $d_1(1)=(1,-m\theta)$.

Consider the vector space $\vf:=\{\xi-\langle
\chi,\xi\rangle, \xi\in\g\}$. The group $\K^\times\times T_0$
naturally acts on this space (we consider the Kazhdan action of $\K^\times$ and the action of
$T_0$ coming from the adjoint action of $T_0$ on $\g$).
So $\vf$ is graded,
$\vf=\bigoplus_{i\in \Z}\vf(i)$, the grading is induced by the
diagonal action of $\K^\times$, i.e., $\vf(i):=\{\xi\in\g| (h-m\theta)\xi= (i-2)\xi\}$.
Put $\A:=\U$. As we explained at the end of Section \ref{SECTION_Walg}, the
product in $\U$ has the required form. Set
$\m:=\widetilde{\m}_\chi$. This subspace satisfies the requirements above (that is,
$\vf_{>0}\subset \m\subset \vf_{\geqslant 0}$ and the eigenvalues of $v_1,\ldots,v_n$ are decreasing).  Furthermore,
$\v_0=\l$ and $\A^0=U(\l)$.

  Note that  $\vf:=\z_\g(e)\oplus V$. So we can set
$\A':=\W_V(\Walg)$, where we write $\W_V(\Walg)$ for $\W_V\otimes \Walg$.  From the choice of $m$
it follows that
 all $\theta$-weight spaces of $\z_\g(e)$ with
positive weights lie in $\bigoplus_{i<0}\vf(i)$. Moreover, note that $\widetilde{\m}\cap
V$ is a lagrangian subspace in $V$. Finally, we remark that $\A'_{\geqslant 0}/\A'^+_{\geqslant 0}$ is
naturally identified with $\W_{V_0}(\Walg^0)$, where $V_0:=V\cap \v_0$. Note that $\m_0$ is contained in $ V_0$
and is a lagrangian subspace there.

\begin{Lem}\label{Lem:7.4}
There is an isomorphism $\Phi:\U^\heartsuit\rightarrow \W_V(\Walg)^{\heartsuit}$ satisfying the
conditions of Proposition \ref{Prop:7.2}. For the extension $\Phi:\U^\wedge\rightarrow
\W_V(\Walg)^\wedge$ we have
\begin{equation}\label{eq:4.1}\Phi^{-1}(\W_V(\I)^\wedge)\cap\U=\I^\dagger,\end{equation} where $\I=\Ann_\Walg(M)$ for $M\in\widetilde{\Ocat}(\theta)$, and $\W_V(\I)^\wedge$ denotes the closure of $\W_V(\I):=\W_V\otimes \I$
in $\W_V(\Walg)^\wedge$.
\end{Lem}
\begin{proof}
The algebras $\A^\heartsuit,
\A'^\heartsuit$ are naturally identified with the quotients
of $\K^\times$-finite parts of $$\K[\g^*]^\wedge_\chi[[\hbar]],
\K[S]^\wedge_\chi[[\hbar]]\widehat{\otimes}_{\K[[\hbar]]}\K[V^*]^\wedge_0[[\hbar]]$$
by $\hbar-1$. So the isomorphism $\Phi_\hbar$ from Proposition \ref{Prop:1} induces a $T_0$-equivariant isomorphism
$\Phi:\A^\heartsuit\rightarrow \A'^\heartsuit$. From the properties of $\Phi_\hbar$ indicated
in Proposition \ref{Prop:1} it follows that $\Phi$ has the required properties.

By \cite{Wquant}, Lemma 3.2.5, there is a natural identification of
$\K[S][\hbar]$ with the Rees algebra $R_\hbar(\Walg)$. So we can consider $R_\hbar(\I)$ as
an ideal  in the quantum algebra $\K[S][\hbar]$. Consider the closure
$\overline{\I}_\hbar$ of $R_\hbar(\I)$ in
$\K[S]^\wedge_\chi[[\hbar]]$. The ideal
$$\overline{\J}_\hbar:=\K[[V^*,\hbar]]\widehat{\otimes}_{\K[[\hbar]]}\overline{\I}_\hbar\subset
\K[[\v^*,\hbar]]$$ is closed, $\K^\times$-stable and
$\hbar$-saturated. By \cite{Wquant}, Proposition 3.2.2, there is a unique
ideal $\I^\ddagger\subset \A$ such that $$R_\hbar(\I^\ddagger)=
R_\hbar(\A)\cap
\Phi_\hbar^{-1}(\K[[V^*,\hbar]]\widehat{\otimes}_{\K[[\hbar]]}\overline{\I}_\hbar).$$
Proposition 3.4.1 from \cite{Wquant} asserts that $\I^\dagger=\I^\ddagger$. So we need to prove that
\begin{equation}\label{eq:4.1.1} \Phi^{-1}(\W_V(\I)^\wedge)\cap\A=\I^\ddagger.\end{equation}

We will prove that \begin{equation}\label{eq:4.1.2}\W_V(\I)^\wedge\cap \W_V(\Walg)^\heartsuit=\overline{\J}_\hbar^{fin}/(\hbar-1),\end{equation}
where $\overline{\J}_\hbar^{fin}$ denotes the $\K^\times$-finite part of $\overline{\J}_\hbar$.
This will imply (\ref{eq:4.1.1}).

To prove (\ref{eq:4.1.2}) we will show that both sides equal $\W_V(\Walg)^\heartsuit\I$. First, let us check this for
the r.h.s. As we checked in \cite{Wquant}, Lemma 3.4.3, $\overline{\J}_\hbar$ is generated by its intersection with $\K[[S,\hbar]]$. Recall that $\overline{\J}_\hbar$ is $\K^\times$-stable. But the $\K^\times$-action we consider
 differ from the Kazhdan one by an action by inner automorphisms, so $\overline{\J}_\hbar\cap
 \K[[S,\hbar]]$ is stable w.r.t the Kazhdan action.  Since the Kazhdan grading on $\K[S]$ is positive,
we see that $\overline{\J}_\hbar$ (and so also $\overline{\J}_\hbar^{fin}$)
is generated by its intersection with $\K[S][\hbar]$. It follows that the r.h.s of (\ref{eq:4.1.2}) is generated
(as an ideal in $\W_V(\Walg)^\heartsuit$) by $\I$.

The proof that the l.h.s. coincides with $\W_V(\Walg)^\heartsuit\I$ boils down to the following two claims:

{\bf Claim 1.} Any ideal in $\W_V(\Walg)^\heartsuit$ is generated by its intersection with
$\Walg$.

{\bf Claim 2.} $\W_V(\I)^\wedge\cap\Walg=\I$ (here the condition that $\I$ is the annihilator
of a module from $\widetilde{\Ocat}(\theta)$ is essential).

Let us prove Claim 1. Let $\J$ be a two-sided ideal in $\W_V(\Walg)^\heartsuit$. Consider the corresponding
ideal $R_\hbar(\J)\subset R_\hbar(\W_V(\Walg)^\heartsuit)=\K[[\vf^*,\hbar]]_{\K^\times-fin}$ and its closure
$\overline{R}_\hbar(\J)\subset \K[[\vf^*,\hbar]]$. Then we can repeat the argument above and obtain
that $\overline{R}_\hbar(\J)$ is generated by its intersection with $\K[S][\hbar]$. This yields
$\J=\W_V(\Walg)^\heartsuit (\J\cap\Walg)$.

Proceed to the proof of Claim 2. We can form the algebras $\Walg^\wedge$ and $\W_V^\wedge$ from $\Walg$ and $\W_V$ using the general construction
explained above, so that $\W_V(\Walg)^\wedge$ is decomposed into the completed tensor product of $\W_V^\wedge$
and $\Walg^\wedge$. The ideal $\W_V(\I)^\wedge$ coincides with $\W_V^\wedge\widehat{\otimes}\I^\wedge$, where
$\I^\wedge$ is the closure of $\I$ in $\Walg^\wedge$. So it remains to check that $\I^\wedge\cap\Walg=\I$.
Recall that $\I=\Ann_\Walg(M)$ for some module $M$ from $\widetilde{\Ocat}(\theta)$. Then $\Walg^\wedge$
acts on $M$, and $\I^\wedge\subset \Ann_{\Walg^\wedge}(M)$. It follows that $\I^\wedge\cap \Walg\subset \I$.
The inverse inclusion is obvious.
\end{proof}

\begin{Rem}
It is not clear at the moment whether (\ref{eq:4.1}) holds without the restriction on $\I$. It looks plausible that any ideal is the annihilator of a module from
$\widetilde{\Ocat}(\theta)$. On the other hand, it may happen that the condition
$\I^\wedge\cap \Walg=\I$ holds for any two-sided ideal $\I$, even if $\I$ is not the annihilator
of a module from $\widetilde{\Ocat}(\theta)$.
\end{Rem}

An isomorphism $\Phi:U(\g)^\heartsuit\rightarrow \W_V(\Walg)^\heartsuit$ from the proof of Lemma \ref{Lem:7.4}
gives rise to  an isomorphism
$\Phi^0:U(\l)^\wedge\rightarrow\W_{V_0}(\Walg^0)^\wedge$ mapping
  $U(\l)^\wedge \m_0$ to $\W_{V_0}(\Walg^0)^\wedge \m_0$. This provides an isomorphism
\begin{equation}\label{eq:Psi}\Psi:\underline{\Walg}=(U(\l)^\wedge/U(\l)^\wedge\m_0)^{\m_0}\rightarrow  \W_{V_0}(\Walg^0)^\wedge/(\W_{V_0}(\Walg^0)^\wedge\m_0)^{\m_0}=\Walg^0 \end{equation}
we need in Theorem \ref{Thm_main}.

Before proceeding further let us make a remark on the isomorphism $\Psi$. We will use this remark in
\cite{Miura}.

\begin{Rem}\label{Rem_two_isomorphisms}
Let us discuss a relation between $\Psi$ and the embeddings $\t_0\hookrightarrow \Walg^0,
\underline{\Walg}$. It turns out that $\Psi$ does not intertwine them but   induces a shift
on $\t_0$.

It follows from assertion (iii) of Proposition \ref{Prop:1} that the isomorphism $\Phi:\U^\heartsuit\rightarrow \W_V(\Walg)^\heartsuit$ intertwines the embeddings of $\t_0$. Let $\iota_\g,\iota_{\Walg}, \iota_V$ denote the
embeddings of $\t_0$ to $\g,\Walg,\W_V$, respectively. Of course, $\iota_\g(\xi)$ is nothing else but $\xi$ itself,
and $\Phi(\iota_\g(\xi))=\iota_\Walg(\xi)+\iota_V(\xi)$.
Let us describe $\iota_V$. Let $\chi_1,\ldots,\chi_m$ denote all characters (with multiplicities) of the
representation of $\t_0$ in the lagrangian subspace $\m\cap V=(\n_+\cap V)\oplus\m_0\subset V$. Let $u_1^+,\ldots,u_m^+$ be the corresponding eigenvectors and let $u_1^-,\ldots,u_m^-\in V$ be such that $\omega_V(u_i^-,u_j^+)=\delta_{ij}$ so that
$\xi.u_i^-=-\langle\chi_i,\xi\rangle u_i^-$. Note that $\langle\chi_i,\theta\rangle\geqslant 0$ for all
$i$ and $\langle\chi_i,\theta\rangle=0$ iff $u_i^+\in \m_0$. Now $$\iota_V(\xi)=\frac{1}{2}\sum_{i=1}^m\langle\chi_i,\xi\rangle(u_i^+u_i^-+u_i^-u_i^+)=
\sum_{i=1}^m \langle\chi_i,\xi\rangle u_i^-u_i^+-\frac{1}{2}\langle \sum_{i=1}^m\chi_i,\xi\rangle.$$
Let $\pi:\Walg_{\geqslant 0}\twoheadrightarrow \Walg^0 $ be the natural projection.
Recall that $\W_V(\Walg)^{\heartsuit 0}$ is naturally identified with
$\W_{V_0}(\Walg^0)$. The image of $\iota_\Walg(\xi)+\iota_V(\xi)$ in $\W_{V_0}(\Walg^0)$ coincides
with
\begin{equation}\label{eq:9.1}\pi(\iota_{\Walg}(\xi))+\sum_{i,\langle\chi_i,\theta\rangle=0} u_i^-u_i^+-\frac{1}{2}\langle \sum_{i=1}^m\chi_i,\xi\rangle.\end{equation}
Now let $\iota_{\l},\iota_{\underline{\Walg}},\iota_{V_0}$ denote the embeddings of $\t_0$ into $\l,\underline{\Walg},\W_{V_0}$, respectively, and
let $\Phi^0$ be an isomorphism $U(\l)^\wedge\rightarrow \W_{V_0}(\underline{\Walg})^\wedge$ similar
to that from \cite{Wquant}, Theorem 1.2.1. Then, by \cite{HC}, Theorem 2.3.1 and Remark 2.3.2, we have $\Phi_\l(\iota_\l(\xi))=
\iota_{\underline{\Walg}}(\xi)+\iota_{\W_{V_0}}(\xi)$. So
\begin{equation}\label{eq:9.2}
\Phi_\l(\iota_\l(\xi))=\iota_{\underline{\Walg}}(\xi)+\sum_{i,\langle\chi_i,\theta\rangle=0} u_i^-u_i^+-\frac{1}{2}\langle \sum_{i,\langle\chi_i,\theta\rangle=0}\chi_i,\xi\rangle.
\end{equation}

Since $\Psi$ is given by (\ref{eq:Psi}),  from (\ref{eq:9.1}),(\ref{eq:9.2}) we see that
$\Psi^{-1}$ maps $\pi(\iota_{\Walg}(\xi))$ to $$\iota_{\underline{\Walg}}(\xi)+\frac{1}{2}\langle\sum_{i:\langle\chi_i,\theta\rangle>0} \chi_i,\xi\rangle$$

In \cite{BGK}, Subsection 4.1, Brundan, Goodwin and Kleshchev also constructed an isomorphism $\Walg^0\rightarrow\underline{\Walg}$
(in the case when $\t=\t_0$). Their isomorphism sends $\pi(\iota_{\Walg}(\xi))$ to $\iota_{\underline{\Walg}}(\xi)-\langle\delta,
\xi\rangle$, where $\delta$ is defined as follows. Pick a Cartan subalgebra  $\h\subset\g$ containing $\t$ and $h$.
Let $\Delta_-$ denote the set of all roots $\alpha$ with $\langle\alpha,\theta\rangle<0$. Then
$$\delta=\frac{1}{2}\sum_{\alpha\in \Delta_-,\langle\alpha,h\rangle=-1}\alpha+\sum_{\alpha\in \Delta_-,\langle\alpha,h\rangle\leqslant -2}\alpha.$$

Let us check that $\delta|_{\t_0}=-\frac{1}{2}\sum_{i,\langle\chi_i,\theta\rangle>0}\chi_i$.

Since $e,h$ are $\t_0$-invariant, the representation theory of $\sl_2$ implies

\begin{equation*}\begin{split}&\delta|_{\t_0}=\frac{1}{2}\left(\sum_{\alpha\in \Delta_-,\langle\alpha,h\rangle=-1}\alpha|_{\t_0}\right)+
\frac{1}{2}\left(\sum_{\alpha\in \Delta_-,\langle\alpha,h\rangle\leqslant -2}\alpha|_{\t_0}+
\sum_{\alpha\in \Delta_-,\langle\alpha,h\rangle\geqslant 2}\alpha|_{\t_0}\right)=\\
&\frac{1}{2}\sum_{\alpha\in \Delta_-}\alpha|_{\t_0}-\frac{1}{2}\sum_{\alpha\in \Delta_-,\langle\alpha,\theta\rangle=0,1}\alpha|_{\t_0}\end{split}\end{equation*}
The last expression is the sum of weights of $\t_0$ in $\n_-\cap V=[\n_-,f]$.
Since $\n_-\cap V$ and $\n_+\cap V$ are dual $\t_0$-modules, we are done.
\end{Rem}

Let us complete the proof of Theorem \ref{Thm_main}. Note that $\widetilde{\Wh}(e,\theta)=\widetilde{\Wh}(\A,\m), \widetilde{\Wh}^{\t_0}(e,\theta)
=\widetilde{\Wh}^{\t_0}(\A,\m)$.
On the other hand, let us construct an equivalence  $\widetilde{\Wh}(\A',\m)\rightarrow\widetilde{\Ocat}(\theta)$.
This functor is given by $\KF':M\mapsto M^{\m\cap V}, M\in \widetilde{\Wh}(\A',\m)$. A quasiinverse functor is given by $N\mapsto \K[\m\cap V]\otimes N, N\in \widetilde{\Ocat}(\theta)$. The claim that these two functors are quasiinverse
follows from the representation theory of Heisenberg Lie algebras, see
the proof of Proposition 3.3.4 in \cite{Wquant}. It follows directly from the construction of
$\KF'$ that
\begin{equation}\label{eq:7.8}
M^{\m}=(\KF'(M))^{\Walg_{>0}}.
\end{equation}

Now we set $\KF:=\KF'\circ \Phi_*$. Let us check that $\K$ has the  required properties.  The equality $\Ann_{\Walg}(\KF(M))^\dagger=\Ann_{\U}(M)$ stems from Lemma \ref{Lem:7.4}.
(\ref{eq:7.8}) and (\ref{eq:7.6})
imply assertion 3.
 Assertion 4 follows from assertion 3 and the adjointness of functors
mentioned in Sections \ref{SECTION_Ocat},\ref{SECTION_Milicic_Soergel}.
The second assertion of the theorem now  follows from Corollary \ref{Cor:7.3}.

\section{Applications}\label{SECTION_Application}
Let us discuss some applications of Theorem \ref{Thm_main}. In \cite{Miura}
we will apply it to the study of one-dimensional representations
of W-algebras. More precisely,  we will give a
criterium for $\dim L^\theta(V)<\infty$  in terms of the
annihilator of $\Psi_*(V)\in \underline{\Walg}$-$\Mod$. In particular, this criterium
will  prove of Conjecture 5.2 from \cite{BGK}.

Then, under some conditions
on $e$, we will get a criterium (in terms of $V$) for a finite dimensional module $L^\theta(V)$
to have dimension 1. More precisely, we will check that whenever $\q$ is semisimple,
the following conditions are equivalent provided $L^\theta(V)$ is finite dimensional:
\begin{itemize}
\item $\dim L^\theta(V)=1$.
\item $\dim V=1$ and $\t$ acts by 0 on $V$ (considered as a $\Walg^0$-module).
\end{itemize}

Since the criterium for $L^\theta(V)$ to be finite dimensional is stated, in a sense, in terms of
$\underline{\Walg}$, we need Remark \ref{Rem_two_isomorphisms}.
The condition that $\q$ is semisimple is fulfilled for all so called rigid nilpotent elements
in exceptional Lie algebras. Together with results of Premet, \cite{Premet4}, on "parabolic
induction" for one-dimensional representations of W-algebras (also reproved in \cite{Miura})
this should allow to complete the proof of Premet's conjecture, \cite{Premet2}, that any
W-algebra has a one-dimensional representation.

Another application, as we learned from Jonathan Brundan, is to finite dimensional irreducible representations of Yangians. For type A, Brundan and Kleshchev identified W-algebras with quotients (truncations) of shifted
Yangians. The latter generalize the usual Yangians introduced by Drinfeld, see \cite{BK1}. Using this
presentation of W-algebras, they classified, \cite{BK2}, all their irreducible finite dimensional representations
(Theorem 7.9) and also all irreducible finite dimensional representations of shifted
Yangians (Corollary 7.10), generalizing results of Drinfeld, \cite{Drinfeld}, on the usual Yangians.
The classification for $W$-algebras is made in terms of some Young diagrams.
On the other hand, using the results announced in the previous paragraph, it is possible to describe
irreducible $\underline{\Walg}$-modules $V$ such that $\dim L^\theta(\Psi_*^{-1}(V))<\infty$
using the  classical results of Joseph, \cite{Joseph}, on combinatorial description of primitive ideals
in $U(\sl_n)$ with given associated variety (which is again stated in terms of some Young diagrams).
In this way one can recover Brundan-Kleshchev classification for W-algebras, see Section 5.2 of
\cite{BGK}. Then, perhaps after some work, one can recover the classification for shifted
Yangians.

In the other classical types Brown identified the W-algebras for  rectangular nilpotent
  elements  $e$ with the truncations of so called  twisted Yangians, see \cite{Brown1}.
  Recall that a nilpotent element in a classical Lie algebra is called rectangular if  all
  the numbers in the corresponding partition are the same. Such an element is always of
  principal Levi type. In \cite{Brown2} Brown used Molev's classification, \cite{Molev}, of irreducible finite dimensional representations for twisted Yangians to classify those for W-algebras (in the rectangular case). The answer
  is given in purely combinatorial terms. On the other hand, \cite{BGK}, Conjecture 5.2, together with results of Barbash and Vogan, \cite{BV}, should make it possible to recover Brown's classification.

{\Small Department of Mathematics, Massachusetts Institute of
Technology, 77 Massachusetts Avenue, Cambridge, MA 02139, USA.

\noindent E-mail address: ivanlosev@math.mit.edu}

\end{document}